\documentclass[12pt]{amsart}
\usepackage{bbm}
\usepackage{graphicx}
\usepackage{mathrsfs}
\usepackage{amsmath, amssymb, amsthm, latexsym}
\usepackage{amssymb,amscd,amsmath}
\usepackage{latexsym}
\usepackage{MnSymbol}
\usepackage{enumitem}
\usepackage[center]{caption}
\usepackage{tikz}
\newcounter{braid}
\newcounter{strands}

\DeclareMathAlphabet{\bsf}{OT1}{cmss}{bx}{n}

\pgfkeyssetvalue{/tikz/braid height}{1cm}
\pgfkeyssetvalue{/tikz/braid width}{1cm}
\pgfkeyssetvalue{/tikz/braid start}{(0,0)}
\pgfkeyssetvalue{/tikz/braid colour}{black}
\pgfkeys{/tikz/strands/.code={\setcounter{strands}{#1}}}

\makeatletter
\def\cross{%
  \@ifnextchar^{\message{Got sup}\cross@sup}{\cross@sub}}

\def\cross@sup^#1_#2{\render@cross{#2}{#1}}

\def\cross@sub_#1{\@ifnextchar^{\cross@@sub{#1}}{\render@cross{#1}{1}}}

\def\cross@@sub#1^#2{\render@cross{#1}{#2}}

\def\render@cross#1#2{
  \def\strand{#1}
  \def\crossing{#2}
  \pgfmathsetmacro{\cross@y}{-\value{braid}*\braid@h}
  \pgfmathtruncatemacro{\nextstrand}{#1+1}
  \foreach \thread in {1,...,\value{strands}}
  {
    \pgfmathsetmacro{\strand@x}{\thread * \braid@w}
    \ifnum\thread=\strand
    \pgfmathsetmacro{\over@x}{\strand * \braid@w + .5*(1 - \crossing) * \braid@w}
    \pgfmathsetmacro{\under@x}{\strand * \braid@w + .5*(1 + \crossing) * \braid@w}
    \draw[braid] \pgfkeysvalueof{/tikz/braid start} +(\under@x pt,\cross@y pt) to[out=-90,in=90] +(\over@x pt,\cross@y pt -\braid@h);
    \draw[braid] \pgfkeysvalueof{/tikz/braid start} +(\over@x pt,\cross@y pt) to[out=-90,in=90] +(\under@x pt,\cross@y pt -\braid@h);
    \else
    \ifnum\thread=\nextstrand
    \else
     \draw[braid] \pgfkeysvalueof{/tikz/braid start} ++(\strand@x pt,\cross@y pt) -- ++(0,-\braid@h);
    \fi
   \fi
  }
  \stepcounter{braid}
}

\tikzset{braid/.style={double=\pgfkeysvalueof{/tikz/braid colour},double distance=1pt,line width=2pt,white}}

\newcommand{\braid}[2][]{%
  \begingroup
  \pgfkeys{/tikz/strands=2}
  \tikzset{#1}
  \pgfkeysgetvalue{/tikz/braid width}{\braid@w}
  \pgfkeysgetvalue{/tikz/braid height}{\braid@h}
  \setcounter{braid}{0}
  \let\sigma=\cross
  #2
  \endgroup
}
\makeatother

\input xypic
\newtheorem{theorem}{Theorem}
\newtheorem{proposition}[theorem]{Proposition}

\newtheorem{lemma}[theorem]{Lemma}


\def\Z{\mathbb{Z}}

\def\Q{\mathbb{Q}}

\def\N{\mathbb{N}}

\def\qed{\hfill$\square$\medskip}

\def\Zpk{\mathbb{Z}/p^{k}}
\def\Zpk1{\mathbb{Z}/p^{k-1}}

\newcommand{\rref}[1]{(\ref{#1})}

\newcommand{\beg}[2]{\begin{equation}\label{#1}#2\end{equation}}
\def\r{\rightarrow}

\def\sl2{\widetilde{SL_{2}(\Z)}}

\title[Tate cohomology of connective k-theory]{Tate cohomology of connected k-theory for
elementary abelian groups revisited}
\author{Po Hu, Igor Kriz and Petr Somberg}
\thanks{The authors acknowledge support by grant GA\,CR P201/12/G028.
Kriz also acknowledges the support of a Simons Collaboration Grant.}

\begin{document}
\maketitle

\begin{abstract}
Tate cohomology (as well as Borel homology and cohomology)
of connective K-theory for $G=(\Z/2)^n$ was completely calculated by
Bruner and Greenlees \cite{bg}. In this note, we essentially redo the calculation
by a different, more elementary method, and we extend it to $p>2$ prime. We also
identify the resulting spectra, which are products of Eilenberg-Mac Lane spectra, and finitely many finite
Postnikov towers. For $p=2$,
we also reconcile our answer completely with the result of \cite{bg}, which is in a different
form, and hence the comparison involves some non-trivial
combinatorics. 
\end{abstract}

\vspace{3mm}
\section{Introduction}

Tate cohomology of finite groups was first discovered in number theory, where it was noticed that
in some statements related to duality, the $0$'th cohomology has to be ``corrected" by 
factoring out the image of the norm map from $0$'th homology (see for example Serre \cite{serre}). 
Considering also the kernel of
the norm map, finite group homology and cohomology, with the $0$'th groups ``corrected", fit into
one distinguished theory which became known as Tate cohomology.

\vspace{3mm}

In stable homotopy theory, the natural question was to find an appropriate definition of Tate cohomology
for generalized cohomology theories with a finite (or compact Lie) group action, i.e. equivariant 
spectra \cite{lms}. This was accomplished in the fundamental paper \cite{gm} by Greenlees and May. 
Generalized Tate cohomology 
has since become an important tool in stable homotopy theory, notably in \cite{gmc,real,hhr}. 
Generalized, in contrast with ordinary, Tate cohomology is a key tool for understanding completion 
theorems, which relate, by completion, the coefficients of an equivariant cohomology theory
to its Borel cohomology. The first known completion theorem was proved by Atiyah and Segal for K-theory \cite{as}. 
Other important cases discovered since
then include the Segal conjecture, proved by G.Carlsson \cite{carlsson}, and a
completion theorem for complex cobordism by Greenlees and May \cite{gmc}. This aspect
of Tate cohomology is not visible 
at all on the original concept of ordinary Tate cohomology simply because ordinary cohomology does
not satisfy a completion theorem.

\vspace{3mm}
The role of generalized Tate cohomology in completion theorems stems from the fact that it is, 
like in the ordinary case,
the cofiber of a norm map \cite{gm} from the Borel homology to the Borel cohomology of a given
equivariant spectrum. All these constructions forget much of the equivariant 
structure, in the sense that a morphism of equivariant spectra, which is an equivalence non-equivariantly,
induces an isomorphism on them. This aspect of Tate cohomology
was, more recently, used in a different context by Nikolaus and 
Scholze \cite{scholze} to simplify much of the theory of topological cyclic homology.

\vspace{3mm}
In the case of K-theory, the Atiyah-Segal completion theorem \cite{as} was reinteterpreted by 
Greenlees and May \cite{gm} to say that the Tate cohomology of K-theory is rational. In connection
with this fact they asked about the structure of Tate cohomology of connective K-theory (in
the complex and orthogonal cases), and
proposed an answer in terms of a certain part of K-theory Tate cohomology. In particular, they
conjectured that it is a wedge sum of Eilenberg-MacLane spectra \cite{dm}, Conjecture 13.4.

\vspace{3mm}
Bruner and Greenlees \cite{bg} computed Borel k-homology, Borel k-coho-mology
and Tate k-cohomology for the group $G=(\Z/2)^r$ for $r>1$. They disproved the Greenlees-May conjecture, showing,
however, that a closely related statement holds for $G=(\Z/2)^r$. In particular, there is a large summand which
is a sum of suspensions of $H\Z/2$,
which is essentially dual in k-Borel homology and cohomology, and both duals are present in Tate cohomology.
The computation of \cite{bg} is non-trivial, using techniques of local algebra, as well as the Adams spectral
sequence. The non-torsion part is also closely related to what was conjectured by Greenlees and May in that
it is a sum of copies of $H\Z_2$, and finitely many finite Postnikov towers.

\vspace{3mm}
The main purpose of this note is to exhibit a different approach to these calculations, using more direct
and elementary methods. We also extend the calculations to $p>2$.
Essentially, we directly use induction on $r$,
smashing (or applying the function spectrum) to one additional copy of $B\Z/p_+$ at a time, 
exploiting the fact that most of the $k$-modules involved have $Ext$-dimension $\leq 1$
(noticing that having this on coefficients implies the same statement in spectral algebra).
Our method gives a closed formulas for the Poincare series of
the torsion part of the calculation. For $p=2$, we prove that this is in fact the same answer as the result
of \cite{bg}, which is a non-trivial combinatorial calculation.

\vspace{3mm}
The present paper is organized as follows. We treat the $p=2$ case first, since it is simpler, and we know it 
first. In Section \ref{sb}, we compute the $k$-Borel homology and cohomology for $p=2$, essentially
using homological algebra, which in this case nicely translates to homological algebra of $k$-modules. In
Section \ref{state}, we compute the norm map for $p=2$, and identify the $k$-Tate cohomology for $p=2$
in Theorem \ref{ttt}. In Section \ref{sp}, we explain what needs to be added to those constructions
for handling the case $p>2$. The computation of $k$-Tate cohomology is presented in Theorem \ref{tptt}.
In Section \ref{sbg}, we give the explicit algebraic comparison of our result for $p=2$ with the result of
\cite{bg}.

\vspace{3mm}

\section{Borel $k$-homology and cohomology of $(\Z/2)^i$}\label{sb}

Let $k$ denote connective complex K-theory $E_\infty$ ring spectrum. We have $k_*=\Z[\beta]$ where $\beta$ is the
Bott element in degree $2$. Consider the $\Z[\beta]$-modules 
$$N=\Z[\frac{\beta}{2}],\;M=N/\Z[\beta].$$
We shall write $Tor$ for $Tor^{\Z[\beta]}$. For any graded $\Z[\beta]$-module $W$, and a power series 
$p(w)=\sum_n a_nw^n\in \N_0[[w,w^{-1}]]$, we shall denote by $Wp(w)$ a sum of copies of shifts of $W$ 
where there are $a_n$ copies of $W[n]$ (i.e. of $W$ shifted up by $n\in \Z$).

\begin{lemma}\label{ltor1}
\begin{enumerate}
\item\label{ltori1}
For every $\Z[\beta]$-module $W$, we have
$$Tor_{\geq 2}(M,W)=Tor_{\geq 2}(N,W)=0.$$
\item\label{ltori2}
We have
$$N\otimes_{\Z[\beta]}N\cong N\oplus \Z/2\frac{w^2}{(1-w^2)^2},\; Tor_1(N,N)=0$$
(Note that the graded module $\Z/2$ must have trivial action of $\beta$.)
\item\label{ltori3}
We have
$$N\otimes_{\Z[\beta]}M\cong \Z/2\frac{w^2}{(1-w^2)^2}, \;Tor_1(M,N)=0.$$
\item\label{ltori4}
We have
$$M\otimes_{\Z[\beta]}M\cong \Z/2\frac{w^4}{(1-w^2)^2},\; Tor_1(M,M)\cong M[2].$$

\end{enumerate}
\end{lemma}

\begin{proof}
We have a free $\Z[\beta]$-resolution of $N$ of the form
\beg{eltor1}{0\r \bigoplus_{i\geq 1}\Z[\beta]\{t_i\}\r\bigoplus_{i\geq 0}\Z[\beta]\{z_i\}
}
where $|z_i|=|t_i|=2i$, 
\beg{eltor2}{
t_i\mapsto 2z_i-\beta z_{i-1}, \; i\geq 1.
}
The augmentation sends 
$$z_i\mapsto (\frac{\beta}{2})^i.$$
Similarly, we have a free $\Z[\beta]$-resolution of $M$ of the form
\beg{eltor3}{0\r \bigoplus_{i\geq 0}\Z[\beta]\{t_i\}\r\bigoplus_{i\geq 0}\Z[\beta]\{z_i\}
}
with \rref{eltor2} and 
\beg{eltor4}{t_0\mapsto z_0.
}
Statement \ref{ltori1} follows.

To prove Statement \ref{ltori2}, apply $?\otimes_{\Z[\beta]}N$ to \rref{eltor1}. We obtain
$$\bigoplus_{i\geq 1}\Z[\frac{\beta}{2}]\{t_i\}\r\bigoplus_{i\geq 0}\Z[\frac{\beta}{2}]\{z_i\},$$
$$t_i\mapsto 2(z_i-\frac{\beta}{2}z_{i-1}), \; i\geq 1.$$
Clearly, this is injective and the cokernel is
$$\Z[\frac{\beta}{2}]\{z_0\}\oplus \bigoplus_{i\geq 1}\Z/2[\frac{\beta}{2}]\{z_i-\frac{\beta}{2}z_{i-1}\},$$
as claimed.

For Statement \ref{ltori3}, consider the short exact sequence of $\Z[\beta]$-modules
\beg{eltor5}{
0\r \Z[\beta]\r N\r M\r 0,
}
and the associated long exact sequence of $Tor(?,N)$. We get
$$\diagram 0\rto & Tor_1(M,N)\rto & N\rto^(.3)\alpha & N\oplus \Z/2\frac{w^2}{(1-w^2)^2}
\rto &N\otimes_{\Z[\beta]} M\rto & 0
\enddiagram$$
where $\alpha$ is the inclusion of the first summand. Statement \ref{ltori3} follows. 

For Statement \ref{ltori4}, consider the long exact sequence obtained by applying $Tor(?,M)$
to \rref{eltor5}. We obtain
$$\diagram
0\rto & Tor_1(M,M)\rto & M
\rto^(.3)\gamma & \Z/2\frac{w^2}{(1-w^2)^2}\rto & M\otimes_{\Z[\beta]}M\rto & 0.
\enddiagram$$
The map $\gamma$ is the identity on 
$$\bigoplus_{i\geq 1}\Z/2\{z_i\},$$
and $0$ otherwise. Realizing that
$$Ker(\gamma)\cong M[2],$$
Statement \ref{ltori4} follows.

\end{proof}

\vspace{3mm}
We now turn to discussing $k$-module structures. For an $E_\infty$-ring spectrum $R$, by a module, 
we shall always mean an $E_\infty$-module. For background information about how
to do commutative algebra over $E_\infty$-ring spectra, the reader
is referred to \cite{ekmm}. 

\vspace{3mm}
By a degree $0,1$-$R_*$-resolution we shall mean 
a free resolution of an $R_*$-module non-trivial only in degrees $0,1$. It is easy to see that a degree
$0,1$-$R_*$-resolution is realized by an $R$-module uniquely up to $R$-module equivalence. 

Thus, we have unique $k$-modules $\mathscr{M},\mathscr{N}$ with $\mathscr{M}_*=M$,
$\mathscr{N}_*=N$ and, furthermore, 
$$k\wedge B\Z/2\sim \mathscr{M}[-1]$$
as $k$-modules. Letting 
$$H=k/(\beta,2),$$
(recall that the order of killing elements does not matter), we further conclude from Lemma \ref{ltor1} that
\beg{eltor6}{\mathscr{M}\wedge_k\mathscr{N}\sim H\frac{w^2}{(1-w^2)^2}
=Hw^2(\sum_{i=0}^\infty w^{2i})^2=H\sum_{i=1}^\infty iw^{2i}.
}
(using the same convention for module spectra as for graded modules). This is because we get a map
to $H[2(i+j+1)]$ by sending the generator $z_{i+1}(\beta/2)^j$ to $H$, while sending the
generators $z_{k}(\beta/2)^\ell$ to $0$ for $(k,\ell)\neq (i+1,j)$. 

Now consider the cofibration sequence
\beg{eltor7}{\mathscr{M}\r\mathscr{N}\wedge_k\mathscr{M}\r\mathscr{M}\wedge_k\mathscr{M}.
}
From the above discussion, we obtain a commutative diagram of cofibrations in the derived
category of $k$-modules
\beg{eltor9}{\diagram
\mathscr{M}[2]\dto\rto^\kappa & H\frac{w^4}{(1-w^2)^2}\dto \rto &\mathscr{M}\wedge_k\mathscr{M}
\dto^{Id}\\
\mathscr{M}\dto\rto & H\frac{w^2}{(1-w^2)^2}\dto\rto &\mathscr{M}\wedge_k\mathscr{M}\\
H\frac{w^2}{1-w^2}\rto ^{Id} & H\frac{w^2}{1-w^2} &
\enddiagram
}
where the map $\kappa$ is $0$ on coefficients. But it follows from the resolution of $\mathscr{M}$ that
all non-trivial maps of $k$-modules $\mathscr{M}\r H[2i]$ are non-zero on coefficients. Thus, $\kappa=0$, 
and we have proved

\begin{proposition}\label{ptor1}
As $k$-modules, we have
$$k\wedge B\Z/2\wedge B\Z/2[2]\sim \mathscr{M}\wedge_k\mathscr{M}\sim \mathscr{M}[3]\vee 
H\frac{w^4}{(1-w^2)^2}.$$
\end{proposition}
\qed

\vspace{3mm}
To go farther, we note the following:

\begin{lemma}\label{lhm}
We have
$$H\wedge_k\mathscr{M}\sim H\frac{w^2}{1-w}.$$
\end{lemma}

\begin{proof}
Apply $?\wedge_kH$ to the geometric realization of \rref{eltor3}. We see that we can eliminate
the $i=0$ terms, since they are linked by an isomorphism. Omitting those terms, we 
claim that the resulting map of $k$-modules
$$H\frac{w^2}{1-w^2}\r H\frac{w^2}{1-w^2}$$
is zero. To this end, it suffices to show that the self-maps $2,\beta$ of $H$ are both zero. But considering
the defining resolution of $H$ by free $k$-modules, we see that the only non-zero self-maps of $H$ as $k$-modules
are in negative degrees. 
\end{proof}

\vspace{3mm}
Let 
$$E=H\frac{w^4}{(1-w^2)^2}$$
be the ``error term" in Proposition \ref{ptor1}.

\begin{proposition}\label{pbh}
The following statements hold as $k$-modules.

\begin{enumerate}
\item\label{pbhi1}
For $i\geq 0$,
$$E\wedge_k\underbrace{\mathscr{M}\wedge_k\dots\wedge_k\mathscr{M}}_{\text{$i$ times}}
=H\frac{w^{2(i+2)}}{(1-w^2)^2(1-w)^i}.$$

\item\label{pbhi2}
For $i\geq 2$,
$$\begin{array}{l}k\wedge\underbrace{B\Z/2\wedge\dots\wedge B\Z/2}_{\text{$i$ times}}[i]=
\underbrace{\mathscr{M}\wedge_k\dots\wedge_k\mathscr{M}}_{\text{$i$ times}}
=\\[3ex]
\displaystyle
\mathscr{M}[3(i-1)]\vee H\frac{w^4}{(1-w^2)^2}\left(\frac{w^2}{1-w}\right)^{i-2}
\left(\frac{(1-w)^{i-1}w^{i-1}-1}{(1-w)w-1}\right).
\end{array}$$

\item\label{pbhi3}
For $i\geq 0$,
$$\begin{array}{l}k\wedge B(\Z/2)^i_+=k\vee
\mathscr{M}w^{-3}((1+w^2)^i-1)\vee\\[3ex]
\displaystyle
H\frac{1}{(1-w^2)^2(1-w)^{i-1}}\left(
\frac{1-(1+w^2)^i(1-w)^i}{1-(1+w^2)(1-w)}-\frac{1-(1-w)^i}{1-(1-w)}
\right).
\end{array}
$$

\end{enumerate}
\end{proposition}

\begin{proof}
We have $k\wedge B\Z/2_+=k\vee \Sigma^{-1}\mathscr{M}$. Statement \ref{pbhi1} follows
immediately from Lemma \ref{lhm}. Statement \ref{pbhi2} follows from Proposition \ref{ptor1} and Lemma
\ref{lhm} by induction, using the formula
$$1+\frac{1}{a}+\dots+\frac{1}{a^{i-2}}=
\frac{a^{i-1}-1}{(a-1)a^{i-2}}.
$$
To prove Statement \ref{pbhi3}, let, as an induction hypothesis,
$$
k\wedge B(\Z/2)^i_+=k\vee \mathscr{M} \cdot p_i\vee H\cdot q_i.
$$
We have proved that $p_1=w^{-1}$, $q_1=0$, and
$$\begin{array}{l}
k\vee \mathscr{M}\cdot p_{i+1}\vee H\cdot q_{i+1}=\\[1ex]
(k\vee \mathscr{M}\cdot p_i\vee H\cdot q_i)\wedge_k(k\vee\mathscr{M} w^{-1})=\\[1ex]
\displaystyle
k\vee \mathscr{M}\cdot(w^{-1}+(1+w^2)p_i)\vee H\cdot\left(
\frac{w^3}{(1-w^2)^2}p_i+\frac{q_i}{1-w}
\right).
\end{array}
$$
Thus, we have
\beg{epbh+}{p_{i+1}=w^{-1}+(1+w^2)p_i,
}
\beg{epbh++}{q_{i+1}=\frac{w^3}{(1-w^2)^2}p_i+\frac{q_i}{1-w}.
}
From \rref{epbh+}, we obtain
$$p_i=((1+w^2)^i-1)w^{-3},$$
so \rref{epbh++} gives
$$q_{i+1}=\frac{(1+w^2)^i-1}{(1-w^2)^2}+\frac{q_i}{1-w}.$$
Solving the recursion gives Statement \ref{pbhi3}.
\end{proof}

\vspace{3mm}
We now calculate $F(B(\Z/2)^i_+,k)$. Again, the restriction map $B\Z/2_+\r S^0$ induces a splitting
$$F(B\Z/2_+, k)\sim k\vee F(B\Z/2,k).
$$
Let 
$$\mathscr{P}=F(B\Z/2,k)=xk[[x]]/(x\beta+2)$$
where
$$k[[x]]=\prod_{i\geq 0}k\{x^i\},\; |x|=-2.$$
This can be expressed, in an obvious way, as a homotopy limit of $k[x]/x^j$ in the obvious sense.
When involving these $k$-modules, we shall denote by $\widehat{\wedge}_k$ the operation of taking 
$\wedge_k$ on the corresponding ``truncated polynomial" modules, and then passing to homotopy limit of 
the resulting inverse sequence. Thus, similarly to Lemma \ref{ltor1}, Proposition \ref{ptor1} and Lemma \ref{lhm}, 
we prove that
$$\mathscr{P}\widehat{\wedge}_k\mathscr{P}\sim \mathscr{P}[-2]\vee H\frac{w^{-4}}{(1-w^{-2})^2},$$
$$\mathscr{P}\widehat{\wedge}_k H\sim H \frac{w^{-1}}{1-w^{-1}}.$$
Analogously to the proof of Proposition \ref{pbh}, we then obtain

\begin{proposition}\label{pbc}
We have
$$\begin{array}{l}
F(B(\Z/2)^i,k)\sim \mathscr{P}\cdot \protect((1+w^{-2})^i-1)\vee\\[3ex]
\displaystyle
H\frac{1}{(1-w^{-2})^2(1-w^{-1})^{i-1}}\left(
\frac{1-(1+w^{-2})^i(1-w^{-1})^i}{1-(1+w^{-2})(1-w^{-1})}-\frac{1-(1-w^{-1})^i}{1-(1-w^{-1})}
\right).
\end{array}
$$
\end{proposition}
\qed

Note the interesting symmetry of the ``error terms" in Borel homology and cohomology. 

\vspace{3mm}
\section{Tate cohomology}\label{state}

Now to compute the Tate cohomology 
\beg{ekktate}{\widehat{k}^{(\Z/2)^n},} 
we use the cofibration sequence of \cite{gm}
$$\diagram
k\wedge B(\Z/2)^n_+\rto^N & F(B(\Z/2)^n_+,k)\rto &\widehat{k}^{(\Z/2)^n},
\enddiagram
$$
and the computation of Proposition \ref{pbc}, and Statement \ref{pbhi3} of Proposition \ref{pbh}.
First, note that there is no possibility of extension of $k$-modules involving any of the $H$ summands, 
since there are no such extensions in non-negative degrees. It follows that \rref{ekktate}, as a $k$-module, is
a wedge sum of copies of $H$ in degrees we already know, and a $k$-module with coefficients $Q$, which
maps, as a $k$-module and in a way injective on coefficients, into $\widehat{K}^{(\Z/2)^n}$, which
has coefficients
\beg{ekktate1}{J^\wedge_2[\beta,\beta^{-1}][\frac{1}{2}]
}
(where $J$ denotes the augmentation ideal of $R((\Z/2)^n)$ and $(?)^\wedge_2$ means
completion at $2$). We will see that $Q$ is a completion 
of a $\Z[\beta]$-module with resolution in degrees $0,1$, and hence, above, the $k$-module is determined
by its coefficients $Q$.

\vspace{3mm}
To calculate $Q$, let $\alpha_1,\dots,\alpha_n$ be $\otimes$-$\Z/2$-independent $1$-dimensional complex
representations of $(\Z/2)^n$. Let $x_i\in k^*(B\Z/2)^n$ be the Euler class of $\alpha_i$. Then we can
write $1+x_i\beta=\alpha_i\in R((\Z/2)^n)$, and
\beg{ekktate2}{\widetilde{Q}=J^\wedge_2[\beta,\beta^{-1}][\frac{1}{2}]=
\bigoplus_{\substack{k\geq 1\\1\leq i_1<\dots<i_k\leq n}}
\mathbb{Q}_2[\beta,\beta^{-1}]\{x_{i_1}\dots x_{i_k}\}.
}
It is also worth noting that \rref{ekktate2} is, in fact,
\beg{ekktate+}{
\begin{array}{l}\displaystyle
K^*((B\Z/2)^n)/K_*\cdot\prod_{i=1}^n (2-\beta x_i)=\\[1ex]
\displaystyle
\Z[\beta,\beta^{-1}][[x_1,\dots,x_n]]/ (x_i(2+\beta x_i),\prod_{i=1}^n(2+\beta x_i)).
\end{array}
}
Now $Q$ is a $k^*((B\Z/2)^n)$-submodule of \rref{ekktate2}, while on a multiple of $x_i$,
$x_i=-\displaystyle\frac{2}{\beta}$. Thus, we have proved

\begin{lemma}\label{lkktate1}
$Q$ is a $\Z_2[\beta,\displaystyle\frac{2}{\beta}]$-submodule of $\widetilde{Q}$.
\end{lemma}
\qed

Let $P\subset Q$ be the image of $k^*B(\Z/2)^n$. Then we have an extension
of $\Z[\beta,\displaystyle\frac{2}{\beta}]$-modules
$$0\r P\r Q\r M\r0.$$

\begin{proposition}\label{pkktate1}
The $\Z[\beta,\displaystyle\frac{2}{\beta}]$-module $P$ is free on generators
\beg{ettgena}{x_{i_1}\dots x_{i_k},\; 1\leq k<n, \; i_1<\dots <i_k,}
\beg{ettgenb}{\prod_{i=1}^{n}(1+\frac{\beta}{2}x_i)-1.
}
\end{proposition}

\begin{proof}
The fact that the elements \rref{ettgena} are in $P$ follows from the fact that $P$ is the image
of $k^*B(\Z/2)^n$. For the same reason, it contains $1$, which is equal to \rref{ettgenb} by 
the relation \rref{ekktate+}. On the other hand, $P$ is generated by the elements \rref{ettgena}, 
\rref{ettgenb} ($=1$), and $x_1\dots,x_n$ which, however, is a $\displaystyle (\frac{2}{\beta})^n$-multiple
of the element \rref{ettgenb}, plus a $\Z_2[\beta,\displaystyle\frac{2}{\beta}]$-linear combination 
of the elements \rref{ettgena}, and thus can be eliminated. On the other hand, the elements
\rref{ettgena}, \rref{ettgenb} are $\Z_2[\beta,\displaystyle\frac{2}{\beta}]$-linearly independent
in $P$ since \rref{ettgena} and $x_1\dots x_n$ are $\Q_2[\beta,\displaystyle\frac{2}{\beta}]$-linearly
independent in $\widetilde{Q}$.
\end{proof}

\vspace{3mm}
Now we have proved that
\beg{ekktate3}{M=\bigoplus_{\substack{1\leq k\leq n\\ 1\leq i_1<\dots<i_k\leq n}}
\Z[\frac{\beta}{2}]/\Z[\beta]y_{i_1\dots i_k}
}
where the element $y_{i_1\dots i_k}$ is in $Tor$-degree $k$ and dimension $2k-2$.

\vspace{3mm}

\begin{proposition}\label{pchoice1}
The elements $y_{i_1\dots i_k}$ can be chosen in such a way that in $Q$, we have
\beg{ekktate4}{
y_{i_1\dots i_k}=\frac{\beta^k}{2}
\left(\prod_{s=1}^{k}(1+\frac{\beta}{2}x_{i_s})-1
\right)\prod_{u=1}^{n-k}(2+\beta x_{j_u})
}
where $\{i_1,\dots,i_k,j_1,\dots,j_{n-k}\}=\{1,\dots,n\}.$
\end{proposition}

\begin{proof}
An induction on $n$. For $n=1$, we know from \cite{gm} that
$$y_1=\frac{\beta^2}{4}x_1=\frac{\beta}{2}((1+\frac{\beta}{2}x_1)-1).$$
Now for $k<n$, the element $y_{i_1\dots i_k}$ can be chosen as a restriction from
$k_*B(\Z/2)^k$ for some subgroup of $(\Z/2)^n$ isomorphic to $(\Z/2)^k$. Thus,
in $Q$, it will be given by corestriction of the same element from the same subgroup, which
is given by multiplication by
$$
\prod_{u=1}^{n-k}(2+\beta x_{j_u}).
$$
Hence, our statement for $y_{i_1\dots i_k}$, $k<n$, follows from the induction hypothesis. For $k=n$,
on the other hand, we can choose 
$$y_{i_1\dots i_n}=-\frac{\beta^n}{2}$$
simply by its dimension and the fact that $Q\subset\widetilde{Q}$ is a submodule (and thus, in particular,
free of torsion in $\beta$ and $2$).
\end{proof}

\vspace{3mm}

Thus, it remains to calculate the extension determined by the relations \rref{ekktate4}. To this end, 
it is helpful to change variables by putting
\beg{ekktate5}{t_i=x_i+\frac{2}{\beta}.
}

\begin{lemma}\label{lchv}
The $\Z_2[\beta,\displaystyle\frac{2}{\beta}]$-module $P$ is free on the generators
\beg{ekktate6}{
\begin{array}{l}
\displaystyle
z_{i_1\dots i_k}=t_{i_1}\dots t_{i_k}-\left(\frac{2}{\beta}\right)^k,\; 1\leq k<n, \; 1\leq i_1<\dots<i_k\leq n,\\[1ex]
\displaystyle
z_{1\dots n}=\left(\frac{\beta}{2}\right)^nt_1\dots t_n -1.
\end{array}
}
\end{lemma}

\begin{proof}
The base change matrix between the generators \rref{ekktate6} and those of Proposition \ref{pkktate1} is
triangular with invertible elements on the diagonal.
\end{proof}

\vspace{3mm}

\begin{proposition}\label{pchv}
The generators $y_{i_1\dots i_k}$ of Proposition \ref{pchoice1} can be changed to generators
$y_{i_1\dots i_k}^\prime$ such that in $Q$, we have
$$y_{i_1\dots i_k}^\prime =\frac{\beta^n}{2}\cdot z_{i_1,\dots i_k},\; 1\leq k\leq n,\; 1\leq
i_1<\dots<i_k\leq n.$$
\end{proposition}

\begin{proof}
By Proposition \ref{pchoice1}, applying the base change, we have in $Q$:
$$y_{i_1\dots i_k}=\frac{\beta^n}{2}\cdot\left( \protect
(\frac{\beta}{2})^k
\prod_{s=1}^{k} t_{i_s}-1
\right)
\prod_{u=1}^{n-k} t_{j_u}.
$$
For $k=n$, this is the desired generator. For $1\leq k<n$, we get the desired generator by 
subtracting $y_{i_1\dots i_k}$ from 
$$\frac{\beta^n}{2}(\frac{2}{\beta})^{n-k}\left(\protect
(\frac{\beta}{2})^n\prod_{s=1}^{n}t_s-1
\right),$$
which is in the image of $\displaystyle (\frac{2}{\beta})^{n-k}y_{1\dots n}$. (Note that for $1\leq k
<n$, the sets $\{i_1<\dots< i_k\}$ are in bijective correspondence with their complements.)
\end{proof}

\vspace{3mm}
Let $Q_n$ be the $\Z_2[\beta,\displaystyle\frac{2}{\beta}]$-submodule of $\Q_2[\beta,\beta^{-1}]$
generated by $1,\beta^{n-1}(\displaystyle \frac{\beta}{2})^i$, $i\in \N$. Then, as remarked,
we have a unique $k$-module $\mathscr{Q}_n$ with coefficients $Q_n$. We have proved

\vspace{3mm}

\begin{theorem}\label{ttt}
For $n\geq 0$, the Tate cohomology \rref{ekktate} is isomorphic, as a $k$-module, to
$$
\mathscr{Q}_n\cdot ((1+w^{-2})^n-w^{-2n})\vee H\cdot f(w)
$$
where 
$$\begin{array}{l}
\displaystyle 
f(w)=
\frac{w}{(1-w^2)^2(1-w)^{n-1}}\left(
\frac{1-(1+w^2)^n(1-w)^n}{1-(1+w^2)(1-w)}-\frac{1-(1-w)^n}{1-(1-w)}
\right)+
\\[3ex]
\displaystyle
\frac{1}{(1-w^{-2})^2(1-w^{-1})^{n-1}}\left(
\frac{1-(1+w^{-2})^n(1-w^{-1})^n}{1-(1+w^{-2})(1-w^{-1})}-\frac{1-(1-w^{-1})^n}{1-(1-w^{-1})}
\right).

\end{array}$$
\end{theorem}
\qed

\vspace{3mm}
The structure of $\mathscr{Q}_n$ as a spectrum is clarified by the following

\vspace{3mm}
\begin{proposition}\label{ppost}
As spectra, we have:
\beg{epost1}{k^\wedge_2[\frac{2}{\beta}]=k^\wedge_2\vee\bigvee_{n\geq 1} H\Z_2[-2n],
}
\beg{epost2}{k^\wedge_2[\frac{2}{\beta},\frac{\beta}{2}]=\bigvee_{n\in \Z}H\Z_2[2n],
}
\beg{epost3}{\mathscr{Q}_n=\bigvee_{i\geq 1} H\Z_2[-2i]\vee
\tau_{\leq 2n-2}k^\wedge_2\vee\bigvee_{i\geq 0} H\Z_2[2n+2i]
}
where $\tau_{\leq m}X$ of a spectrum $X$ denotes the part of the Postnikov tower up to and
including the $m$'th homotopy group.
\end{proposition}

\begin{proof}
For \rref{epost1}, we have a map of $k$-modules
\beg{epost+}{
k^\wedge_2[\frac{2}{\beta}]\r (k^\wedge_2\vee\bigvee_{n\geq 1} H\Z[-2n])\wedge M\Z/2,
}
mapping to each summand by killing the other generators. In spectra, the map \rref{epost+} survives
the $d_1$ of the corresponding Bockstein spectral sequence, which has no higher differentials (since the
$E_2$-term is concentrated in even degrees),
thus proving \rref{epost1}. Then \rref{epost2} follows by localization. For \rref{epost3}, it then suffices
to remark that we have a homotopy pushout of $k$-modules (and hence spectra)
$$\diagram
k^\wedge_2[\frac{2}{\beta}][2n-2]\dto_{\beta^{n-1}}\rto &
k^\wedge_2[\frac{2}{\beta},\frac{\beta}{2}][2n-2]\dto \\
k^\wedge_2[\frac{2}{\beta}]\rto & \mathscr{Q}_n.
\enddiagram
$$

\end{proof}

\vspace{3mm}
\section{The case of $p>2$}\label{sp}

The case of $G=(\Z/p)^n$ for $p>2$ prime is directly analogous (or more precisely a direct generalization).
Nevertheless, it still introduces some new effects. For one thing, while not necessary, it is convenient to localize
at $p$, which is possible in the category of $E_\infty$-ring spectra, so we have an $E_\infty$-ring spectrum
$k_{(p)}$. We may start out with this spectrum, since the Tate cohomology is $p$-complete anyway.
The reason it is convenient to localize is that then we can make the multiplicative group $p$-typical,
with $p$-series 
$$[p]x=px +\beta^{p-1}x^p$$
(where $[p]x$ denotes the $p$-series).
For Borel homology and cohomology, the basic building
blocks are the $\Z_{(p)}[\beta]$-modules
$$N=\Z_{(p)}[\beta,\frac{\beta^{p-1}}{p}],$$
$$M=(N/\Z_{(p)}[\beta])[2(2-p)],$$
$$\widetilde{H}=\Z_{(p)}[\beta]/(p,\beta^{p-1}).$$
The reason for the shift is that that way, the bottom degree element is in dimension $2$, and other summands we
encounter are shifts in the positive dimension, which seems more natural.

The modules $M,N$ have resolutions in degrees $0,1$, so they are uniquely realized by $k_{(p)}$-modules
$\mathscr{M}$, $\mathscr{N}$. The module $\widetilde{H}$ is realized by the $k_{(p)}$-module
$\widetilde{\mathscr{H}}=k_{(p)}/(p,\beta^{p-1})$ (i.e. by killing a regular sequence). We then have
\beg{epborel1}{k_{(p)}\wedge B\Z/p_+=k_{(p)}\vee \mathscr{M}[-1](1+w^2+\dots w^{2(p-2)}).
}
Directly analogously to Lemma \ref{ltori1}, one proves

\begin{lemma}\label{lptori1}
\begin{enumerate}

\item\label{lptori2}
We have
$$N\otimes_{\Z_{(p)}[\beta]}N\cong N\oplus \widetilde{H}\frac{w^{2(p-1)}}{(1-w^{2(p-1)})^2},\; Tor_1(N,N)=0.$$

\item\label{lptori3}
We have
$$N\otimes_{\Z_{(p)}[\beta]}M\cong \widetilde{H}\frac{w^{2(p-1)}}{(1-w^{2(p-1)})^2}, \;Tor_1(M,N)=0.$$
\item\label{lptori4}
We have
$$M\otimes_{\Z_{(p)}[\beta]}M\cong \widetilde{H}\frac{w^4}{(1-w^{2(p-1)})^2},\; Tor_1(M,M)\cong M[2].$$

\end{enumerate}
\end{lemma}

\qed

\vspace{3mm}
On $k_{(p)}$-modules, one then gets
$$\mathscr{M}\wedge_{k_{(p)}}\mathscr{M}=\mathscr{M}[3]\vee \widetilde{\mathscr{H}}
\frac{w^4}{(1-w^{2(p-1)})^2},
$$
$$\mathscr{M}\wedge_{k_{(p)}}\widetilde{\mathscr{H}}=\widetilde{\mathscr{H}}\frac{w^2(1+w)}{1-w^{2(p-1)
}}.$$
Using this, putting
$$k_{(p)}\wedge B(\Z/p)^i_+=k_{(p)}\vee \mathscr{M}p_i\vee\widetilde{\mathscr{H}}q_i,$$
one has $p_1=w^{-1}(1+w^2+\dots+w^{2(p-2)})$, $q_1=0$, and computes
$$\begin{array}{l}
(k_{(p)}\vee\mathscr{M}p_i\vee\widetilde{\mathscr{H}}q_i))\wedge_{k_{(p)}}
(k_{(p)}\vee \mathscr{M}w^{-1}(1+w^2+\dots+w^{2(p-2)}))=\\[2ex]
k_{(p)}\vee\mathscr{M}((1+w^2+\dots+w^{2(p-1)})p_i+w^{-1}(1+w^2+\dots+w^{2(p-2)}))\vee\\[2ex]
\displaystyle
\widetilde{\mathscr{H}}\left(
p_i\frac{w^{3}}{(1-w^{2(p-1)})(1-w^2)}+\frac{q_i}{1-w}
\right).
\end{array}
$$
This gives
$$p_i=((1+w^2+\dots+w^{2(p-1)})^i-1)w^{-3},$$
$$q_{i+1}=\frac{(1+w^2+\dots+w^{2(p-1)})^i-1}{(1-w^{2(p-1)})(1-w^2)}+\frac{q_i}{1-w}.$$
Solving the recursion, we obtain

\begin{proposition}\label{ppbh}
We have
$$
\begin{array}{l}k_{(p)}\wedge B(\Z/p)^i_+=k_{(p)}\vee
\mathscr{M}w^{-3}((1+w^2+\dots+w^{2(p-1)})^i-1)\vee\\[3ex]
\displaystyle
\widetilde{\mathscr{H}}\frac{1}{(1-w^{2(p-1)})(1-w^2)(1-w)^{i-1}}\cdot\\[3ex]
\displaystyle
\left(
\frac{1-(1+w^2+\dots+w^{2(p-1)})^i(1-w)^i}{1-(1+w^2+\dots+w^{2(p-1)})(1-w)}-\frac{1-(1-w)^i}{1-(1-w)}
\right).
\end{array}
$$
\end{proposition}
\qed

\vspace{3mm}
The situation in Borel cohomology is again analogous. We have
$$F(B\Z/p,k_{(p)})=\mathscr{P}(1+w^{-2}+\dots+w^{-2(p-2)})$$
where 
$$\mathscr{P}=xk_{(p)}[[x^{p-1}]]/((x\beta)^{p-1}+p).$$
Again, we have
$$\mathscr{P}\widehat{\wedge}_{k_{(p)}}\mathscr{P}=\mathscr{P}[-2]\vee \widetilde{\mathscr{H}}
\frac{w^{-4}}{(1-w^{2(p-1)})^2}$$
and
$$\mathscr{P}\widehat{\wedge}_{k_{(p)}}\widetilde{\mathscr{H}}=\widetilde{\mathscr{H}}
\frac{w^{-1}(1+w^{-1})}{1-w^{-2(p-1)}},
$$
Putting these together, we obtain the following result.

\begin{proposition}\label{ppbc}
We have
$$\protect
\begin{array}{l}F(B(\Z/p)^i,k_{(p)})=\protect
\mathscr{P}w^{-1}((1+w^{-2}+\dots+w^{-2(p-1)})^i-1)\vee\\[3ex]\protect
\displaystyle
\widetilde{\mathscr{H}}\frac{1}{(1-w^{-2(p-1)})(1-w^{-2})(1-w^{-1})^{i-1}}\cdot\\[3ex]
\protect\displaystyle\left(
\frac{1-(1+w^{-2}+\dots+w^{-2(p-1)})^i(1-w^{-1})^i}{1-(1+w^{-2}+\dots+w^{-2(p-1)})(1-w^{-1})}\protect
-\frac{1-(1-w^{-1})^i}{1-(1-w^{-1})}
\right).
\end{array}
$$
\end{proposition}
\qed

\vspace{3mm}
To calculate Tate cohomology, we use again the cofibration sequence
$$\diagram
k_{(p)}\wedge B(\Z/p)^n_+\rto^N & F(B(\Z/p)^n_+,k_{(p)})\rto &\widehat{k}^{(\Z/p)^n},
\enddiagram
$$
together with the input from Proposition \ref{ppbh}, \ref{ppbc}. Again, for dimensional reasons, there
is no room for extensions involving the copies of $\widetilde{\mathscr{H}}$. Thus, again, it suffices
to compute the image $Q$ of $\widehat{k}^{(\Z/p)^n}$ in 
$$\widetilde{Q}=\widehat{K}^{(\Z/p)^n}=J^\wedge_p[\beta,\beta^{-1}][\frac{1}{p}]=
\bigoplus_{\substack{k\geq 1\\1\leq i_1<\dots<i_k\leq n\\1\leq\epsilon_i\leq p-1}}
\mathbb{Q}_p[\beta,\beta^{-1}]\{x_{i_1}^{\epsilon_1}\dots x_{i_k}^{\epsilon_k}\}
$$
where $J$ is, again, the augmentation ideal of $R(\Z/p)^n$.
Again, this is also equal to
\beg{epkktate+}{
\begin{array}{l}\displaystyle
K^*((B\Z/p)^n)/K_*\cdot\prod_{i=1}^n (p-(\beta x_i)^{p-1})=\\[1ex]
\displaystyle
\Z[\beta,\beta^{-1}][[x_1,\dots,x_n]]/ (x_i(p+(\beta x_i)^{p-1}),\prod_{i=1}^n(p+(\beta x_i)^{p-1})).
\end{array}
}
Again, $Q$ is a $k^*_{(p)}(B(\Z/p)^n)$-submodule of \rref{epkktate+}, while on a multiple of $x_i$,
$x_i^{p-1}=-\displaystyle\frac{p}{\beta^{p-1}}$. Thus, we have proved

\begin{lemma}\label{lpkktate1}
$Q$ is a $\Z_p[\beta,\displaystyle\frac{p}{\beta^{p-1}}]$-submodule of $\widetilde{Q}$.
\end{lemma}
\qed

Let again $P\subset Q$ be the image of $k^*_{(p)}B(\Z/p)^n$. Then we have an extension
of $\Z[\beta,\displaystyle\frac{p}{\beta^{p-1}}]$-modules
$$0\r P\r Q\r M\r0.$$
Analogously to Proposition \ref{pkktate1}, we then have

\begin{proposition}\label{ppkktate1}
The $\Z_p[\beta,\displaystyle\frac{p}{\beta^{p-1}}]$-module $P$ is free on generators
\beg{ettpgena}{\begin{array}{l}x_{i_1}^{\epsilon_1}\dots x_{i_k}^{\epsilon_k}x_{j_1}^{p-1}
\dots x_{j_{n-k}}^{p-1},\; 1\leq k< n,0\leq \epsilon_s\leq p-2 \;\text{for $1\leq s\leq k$},\\
i_1<\dots <i_k,\; j_1<\dots<j_{n-k}, \\
\{i_1,\dots,i_k,j_1,\dots,j_{n-k}\}=\{1,\dots,n\},
\end{array}}
\beg{ettpgenb}{(\prod_{i=1}^{n}(1+\frac{\beta^{p-1}}{p}x_i^{p-1})-1)
x_{i_1}^{i_1}\dots x_{i_k}^{\epsilon_k},\; 1\leq k\leq n, 0\leq \epsilon_s\leq p-2.
}
\end{proposition}

\qed

\vspace{3mm}
\vspace{3mm}
Now we have proved that
\beg{epkktate3}{M=\bigoplus_{\substack{1\leq k\leq n\\ 1\leq i_1<\dots<i_k\leq n\\0\leq \epsilon_s\leq p-2}}
\Z[\frac{\beta^{p-1}}{p}]/\Z[\beta]\{y_{i_1\dots i_k}x_{i_1}^{\epsilon_1}\dots x_{i_k}^{\epsilon_k}\}
}
where the generator $y_{i_1\dots i_k}$ is in $Tor$-degree $k$ and dimension $(2k-2)(p-1)$.
Analogously to Proposition \ref{pchoice1}, one has

\vspace{3mm}

\begin{proposition}\label{ppchoice1}
The elements $y_{i_1\dots i_k}$ can be chosen in such a way that in $Q$, we have
\beg{epkktate4}{
y_{i_1\dots i_k}=\frac{\beta^{k(p-1)}}{2}
\left(\prod_{s=1}^{k}(1+\frac{\beta^{p-1}}{p}x_{i_s}^{p-1})-1
\right)\prod_{u=1}^{n-k}(p+\beta x_{j_u}^{p-1})
}
where $\{i_1,\dots,i_k,j_1,\dots,j_{n-k}\}=\{1,\dots,n\}.$
\end{proposition}
\qed

\vspace{3mm}
Now put
$$t_i=x_i^{p-1}+\frac{p}{\beta^{p-1}}.$$
The analogue of Lemma \ref{lchv} is

\begin{lemma}\label{lpchv}
The $\Z_p[\beta,\displaystyle\frac{p}{\beta^{p-1}}]$-module $P$ is free on the generators
\beg{epkktate6}{
\begin{array}{l}
\displaystyle
z_{i_1\dots i_k}x_{i_1}^{\epsilon_1}\dots x_{i_k}^{\epsilon_k}=
(t_{j_1}\dots t_{j_{n-k}}-\left(\frac{p}{\beta^{p-1}}\right)^k)x_{i_1}^{\epsilon_1}\dots x_{i_k}^{\epsilon_k},\\
1\leq k<n, \; 1\leq i_1<\dots<i_k\leq n,0\leq \epsilon_s\leq p-2\;\text{for $1\leq s\leq k$}\\
j_1<\dots<j_{n-k},\; \{i_1,\dots,i_k,j_1,\dots,j_{n-k}\}=\{1,\dots,n\}\\[2ex]
\displaystyle
z_{1\dots n}x_{1}^{\epsilon_1}\dots x_{n}^{\epsilon_n}=(\left(\frac{\beta^{p-1}}{p}\right)^nt_1\dots t_n -1),
0\leq \epsilon_s\leq p-2
\end{array}
}
\end{lemma}
\qed

\vspace{3mm}
Now analogously to Proposition \ref{pchv}, we have

\begin{proposition}\label{ppchv}
The generators $y_{i_1\dots i_k}$ of Proposition \ref{pchoice1} can be changed to generators
$y_{i_1\dots i_k}^\prime$ such that in $Q$, we have
$$y_{i_1\dots i_k}^\prime =\frac{\beta^{n(p-1)}}{p}\cdot z_{i_1\dots i_k},\; 1\leq k\leq n,\; 1\leq
i_1<\dots<i_k\leq n.$$
\end{proposition}
\qed

\vspace{3mm}
Let $Q_n$ be the $\Z_p[\beta,\displaystyle\frac{p}{\beta^{p-1}}]$-submodule of $\Q_p[\beta,\beta^{-1}]$
generated by $1,\beta^{(p-1)(n-1)}(\displaystyle \frac{\beta^{p-1}}{p})^i$, $i\in \N$. Then, as remarked,
we have a unique $k$-module $\mathscr{Q}_n$ with coefficients $Q_n$. We have proved

\vspace{3mm}

\begin{theorem}\label{tptt}
For $n\geq 0$, the Tate cohomology $\widehat{k}^{(\Z/p)^n}$ is isomorphic, as a $k$-module, to
$$
\mathscr{Q}_n\cdot ((1+w^{-2}+\dots+w^{-2(p-1)})^n-w^{-2n(p-1)})\vee \widetilde{\mathscr{H}}\cdot f(w)
$$
where 
$$\begin{array}{l}
\displaystyle 
f(w)=
\frac{w}{(1-w^{2(p-1)})(1-w^2)(1-w)^{n-1}}\cdot\\[3ex]
\displaystyle
\left(
\frac{1-(1+w^2+\dots+w^{2(p-1)})^n(1-w)^n}{1-(1+w^2+\dots +w^{2(p-1)})(1-w)}-\frac{1-(1-w)^n}{1-(1-w)}
\right)+
\\[3ex]
\displaystyle
\frac{1}{(1-w^{-2(p-1)})(1-w^{-2})(1-w^{-1})^{n-1}}\cdot\\[3ex]
\displaystyle
\left(
\frac{1-(1+w^{-2}+\dots+w^{-2(p-1)})^n(1-w^{-1})^n}{1-(1+w^{-2}+
\dots + w^{-2(p-1)})(1-w^{-1})}-\frac{1-(1-w^{-1})^n}{1-(1-w^{-1})}
\right).

\end{array}$$
\end{theorem}
\qed

\vspace{3mm}
As spectra, analogously to Proposition \ref{ppost}, one sees that $\widetilde{\mathscr{H}}$ is
a wedge of suspensions of $H\Z/p$, while, denoting by $\ell$ the Adams summand of $k^\wedge_p$, 
$\mathscr{Q}_n$ is
$$
\bigvee_{i\geq 1} H\Z_p[-2i]\wedge \tau_{\leq 2(p-1)(n-1)}k^\wedge_p\vee\bigvee_{i\geq 0} H\Z_p[2(p-1)n+2i].
$$
The spectrum $\tau_{\leq 2(p-1)(n-1)}k^\wedge_p$, of course, further decomposes as a wedge of suspensions
of $\tau_{\leq 2(p-1)(n-1)}\ell$ where $\ell$ is the Adams summand of $k^\wedge_p$.

\vspace{3mm}
\section{Relation with the calculation of Bruner and Greenlees}\label{sbg}

In this Section, we prove that our computation of Borel homology, Borel cohomology and Tate cohomology for $p=2$
agrees with the result of Bruner and Greenlees \cite{bg}. In fact, the only non-trivial
part left concerns the coefficient of the $H$ summand in Proposition \ref{pbh}, part \ref{pbhi3} for Borel 
homology, Proposition \ref{pbc} for Borel cohomology, and Theorem \ref{ttt} for Tate cohomology. In fact,
it suffices to compute the case of Borel homology, since both \cite{bg} and our paper have a duality which
determines the other two cases in the same way. 

Bruner and Greenlees \cite{bg} define
\begin{equation}
[T_i] = \frac{(-t)^{4-i}[(1-x)^r_{[i]}-x^{1-i}(1-x^2)^r_{[i]}]}{(1-x)^{r+1}}, 
\label{theirs}
\end{equation}
$x=t^2$.
By \cite{bg}, Lemma 4.8.3, the Poincare series of the coefficients of the $H$-summand of $k\wedge B(\Z/2)^r_+$,
in the variable $t$, is
\begin{equation}
[T_2] + [T_3] + \cdots + [T_r] . 
\label{bgmain}
\end{equation}

We will separately compute 
\begin{equation}
\sum_{i=2}^{r} (-t)^{4-i}(1-x)^r_{[i]}
\label{part1}
\end{equation}
and
\begin{equation}
\sum_{i=2}^{r} (-t)^{4-i}x^{1-i}(1-x^2)^r_{[i]}.
\label{part2}
\end{equation}
\vspace{2mm}

For formula (\ref{part1}), we have $r \choose i$ occurring in $[T_2], \ldots, [T_i]$. For $2\leq j \leq i$, the term in $[T_j]$ is
\[ (-t)^{4-j} {r \choose i} (-x)^{i} =(-1)^j t^{4-j} {r \choose i}(-1)^i t^{2i} = {r \choose i} (-1)^{i+j} t^{2i+4-j} . \]

Thus, letting $j$ run from $2$ to $i$, the coefficient of $r \choose i$ in (\ref{part1}) is 
\[
 t^{i+4} - t^{i+5} + \cdots +(-1)^{i-2} t^{2i+2}  = t^{i+4} (1-t+t^2 - \cdots +(-1)^{i-2} t^{i-2}) 
=  t^{i+4} \cdot \frac{1-(-t)^{i-1}}{1+t} . \]

Now sum over $i = 2, \ldots, r$, we get (\ref{part1}) is 
\begin{equation*}
\begin{split}
\frac{1}{1+t} \sum_{i=2}^{r} {r \choose i} t^{i+4}(1-(-t)^{i-1}) & = 
\frac{1}{1+t} \left[ \sum_{i=2}^r {r \choose i} t^{i+4} - \sum_{i=2}^r {r \choose i } (-1)^{i-1}t^{2i+3} \right] \\
& = \frac{1}{1+t} \left[ t^4 \sum_{i=2}^r {r \choose i} t^i + t^3\sum_{i=2}^r {r \choose i}(-1)^{i} t^{2i} \right] \\
& = \frac{1}{1+t} \left[ t^4[(1+t)^r-rt -1] +t^3[(1-t^2)^r -r(-t^2) -1] \right] \\
& = \frac{1}{1+t} \left[ t^4(1+t)^r -rt^5 -t^4 +t^3(1-t^2)^r +rt^5 -t^3 \right] \\
& =  \frac{1}{1+t} \left[ t^4(1+t)^r-t^4 +t^3(1-t^2)^r -t^3 \right].
\end{split}
\end{equation*}

Hence, we get that formula (\ref{part1}) is 
\[ \frac{t^3}{1+t} \left[ t((1+t)^r -1) + ((1-t^2)^r -1)\right] . \]
\vspace{5mm}

Now for formula (\ref{part2}), again, $r \choose i$ occurs in $[T_2], \ldots, [T_i]$. For $2 \leq j \leq i$, the term containing 
$r \choose i$ in $[T_j]$ is 

\[ (-1)^{4-j}x^{1-j}{r \choose i} (-x^2)^i = (-1)^jt^{4-j}t^{2-2j} {r \choose i} (-1)^i t^{4i} 
= {r \choose i} (-1)^{i+j}t^{4i-3j+6} . \]
Summing over $2 \leq j \leq i$, we get the coefficient of $r \choose i$ in (\ref{part2}) is 
\[ (-1)^i t^{4i} + (-1)^{i-1} t^{4i-3} + \cdots + t^{i+6} . \]
(The lowest power of $t$ is always positive.) This is in turn 
\[ t^{i+6}(1-t^3 + t^{6} + \cdots + (-1)^i t^{3i-6}) = t^{i+6} \cdot \frac{1-(-t^3)^{i-1}}{1+t^3} . \]

Hence, (\ref{part2}) is 
\begin{equation*}
\begin{split}
& \frac{1}{1+t^3}\sum_{i=2}^r {r \choose i} t^{i+6} \left[ 1- (-t^3)^{i-1} \right] \\
& =
\frac{1}{1+t^3} \left[ \sum_{i=2}^r {r \choose i}t^{i+6} -
\sum_{i=2}^r {r \choose i} t^{i+6} (-t^3)^{i-1}  \right] \\
& = \frac{1}{1+t^3} \left[ t^6 \sum_{i=2}^r {r \choose i}t^i + \sum_{i=2}^r {r \choose i} (-1)^i t^{4i+3} \right] \\
& = \frac{1}{1+t^3} \left[ t^6 \sum_{i=2}^r {r \choose i} t^i +t^3 \sum_{i=2}^r {r \choose i} (-t^4)^i \right] \\
& = \frac{1}{1+t^3} \left[ t^6 [(1+t)^r -rt -1] + t^3[(1-t^4)^r -r(-t^4) -1] \right] \\
& = \frac{t^3}{1+t^3} \left[ t^3((1+t)^r-1) +((1-t^4)^r -1) \right] . 
\end{split}
\end{equation*}

Putting it together, we get that (\ref{bgmain}) is 
\begin{equation}
 \frac{1}{(1-t^2)^{r+1}} \begin{bmatrix}
 \frac{t^3}{1+t}[t((1+t)^r-1)+((1-t^2)^r-1)] \\
 -\frac{t^3}{1+t^3}[t^3((1+t)^r-1)+((1-t^4)^r-1)] \end{bmatrix} . 
\label{answer}
\end{equation}
\vspace{2mm}

It seems better to set $u = -t$, so we get that (\ref{answer}) is 
\begin{equation}
\frac{1}{(1-u^2)^{r+1}} \begin{bmatrix} 
\frac{u^4}{1-u}((1-u)^r -1)  - \frac{u^3}{1-u}((1-u^2)^r-1) \\
-\frac{u^6}{1-u^3}((1-u)^r-1) +\frac{u^3}{1-u^3}((1-u^4)^r -1) \end{bmatrix}
\label{withneg}
\end{equation}

In (\ref{withneg}), the two terms of the first column in the bracket add up to 
\begin{equation}
 \left( \frac{u^4}{1-u} - \frac{u^6}{1-u^3} \right)[(1-u)^r-1] = 
\frac{u^4(1+u)}{1-u^3} [(1-u)^r -1]. 
\label{column1}
\end{equation}

The two terms of the second column inside the bracket of (\ref{withneg}) add up to 
\begin{equation*}
\begin{split}
& \frac{u^3}{1-u^3}[(1-u^4)^r-1] - \frac{u^3}{1-u} [(1-u^2)^r-1]\\
& = \frac{u^3}{1-u^3}[(1-u^2)^r(1+u^2)^r -1] - \frac{u^3(1+u +u^2)}{1-u^3}[(1-u^2)^r-1] \\
& = \frac{u^3}{1-u^3}[(1+u^2)^r-1](1-u^2)^r - \frac{u^4(1+u)}{1-u^3}[(1-u^2)^r-1]
\end{split}
\end{equation*}

Combining the second term of this with (\ref{column1}), we get 
\begin{equation*}
\begin{split}
& -\frac{u^4(1+u)}{1-u^3}[(1-u)^r(1+u)^r-1 -(1-u)^r +1] \\
& = -\frac{u^4(1+u)}{1-u^3}[(1+u)^r-1](1-u)^r . 
\end{split}
\end{equation*}

Hence, (\ref{withneg}) becomes 
\begin{equation*}
\begin{split}
& \frac{1}{(1-u^2)^{r+1}}\left[ \frac{u^3}{1-u^3}[(1+u^2)^r-1](1-u^2)^r -
\frac{u^4(1+u)}{1-u^3}[(1+u)^r-1](1-u)^r \right] \\
& = \frac{u^3}{1-u^3} \left[ \frac{(1+u^2)^r-1}{1-u^2} -\frac{1}{(1-u^2)^{r+1}}u(1+u)[(1+u)^r-1](1-u)^r \right] \\
& = \frac{u^3}{1-u^3} \left[ \frac{(1+u^2)^r-1}{1-u^2}  -\frac{u[(1+u)^r-1]}{(1+u)^r(1-u)}
\right].
\end{split}
\end{equation*}

If we put this back into $t$, we get that (\ref{bgmain}) is 
\begin{equation}
-\frac{t^3}{1+t^3} \left[ \frac{(1+t^2)^r-1}{1-t^2} + \frac{t[(1-t)^r-1]}{(1-t)^r(1+t)} \right]. 
\label{bganswer}
\end{equation}
\vspace{5mm}

By Proposition \ref{pbh}, part \ref{pbhi3}, the coefficient of the $H$-summand of the
Borel homology spectrum $k\wedge B(\Z/2)_+^r$ is
\begin{equation}
\frac{1}{(1-w^2)^2(1-w)^{r-1}}\left( \frac{1-(1+w^2)^r(1-w)^r}{1-(1+w^2)(1-w)} - 
\frac{1-(1-w)^r}{1-(1-w)} \right)  . 
\label{ours}
\end{equation}
(We continue using $w$ instead of $t$ for the formula in our paper, to keep the notation 
distinct from that of Bruner-Greenlees.) 

The Bruner-Greenlees formula needs to be applied with what 
they call $Start(2)$, i.~e. shift so that the lowest term is $t^2$, which matches what we get from formula (\ref{ours}). 

In fact, in Bruner-Greenlees \cite{bg}, proof of Lemma 4.8.3, they compute explicitly 
\[ (1-t^2)^r[T_i] = {r \choose i}t^{i+4} -{r \choose i+1}(t^{i+6} + t^{i+8}) + 
{r \choose i+2}(t^{i+8} + t^{i+10} + t^{i+12}) - \cdots \]
So the lowest term in $[T_2] + \cdots [T_r]$ is always ${r \choose 2} t^6$ from $[T_2]$. Hence, $Start(2)$ on 
it is just multiplying by $t^{-4}$. Applying this to (\ref{bganswer}), we get 
\begin{equation}
\frac{-1}{t(1+t^3)} \left[ \frac{(1+t^2)^r-1}{1-t^2} + \frac{t[(1-t)^r-1]}{(1-t)^r(1+t)} \right]. 
\label{shifted}
\end{equation}
\vspace{4mm}

\begin{lemma}
Formulas (\ref{ours}) and (\ref{shifted}) agree (up to replacing $t$ by $w$).
\end{lemma}

\begin{proof}
To show this, note that (\ref{shifted}) is 
\begin{equation*}
\begin{split}
& \frac{-1}{t(1+t^3)}
\left[ \frac{(1+t^2)^r-1}{1-t^2} + \frac{t[(1-t)^r-1]}{(1-t^2)(1-t)^{r-1}} \right] \\
& = \frac{-1}{t(1+t^3)} \left[ \frac{[(1+t^2)^r-1](1-t)^{r-1} + t[(1-t)^r-1]}{(1-t^2)(1-t)^{r-1}} \right] \\
& = \frac{-1}{t(1-t+t^2)(1+t)} \left[ \frac{[(1+t^2)^r-1](1-t)^{r-1} + t[(1-t)^r-1]}{(1-t^2)(1-t)^{r-1}} \right] \\
& = \frac{-1}{t(1-t+t^2)}  \left[ \frac{[(1+t^2)^r-1](1-t)^{r-1} + t[(1-t)^r-1]}{(1-t^2)^2(1-t)^{r-2}} \right] \\
& = \frac{-1}{t(1-t+t^2)}  \left[ \frac{[(1+t^2)^r-1](1-t)^{r} + t(1-t)[(1-t)^r-1]}{(1-t^2)^2(1-t)^{r-1}} \right].
\end{split}
\end{equation*}

On the other hand, for our formula (\ref{ours}), let $a = 1+w^2$, $b = 1-w$. The part inside the brackets of (\ref{ours}) is 
\[ \frac{1-a^rb^r}{1-ab} - \frac{1-b^r}{1-b} = \frac{(1-a^rb^r)(1-b) -(1-ab)(1-b^r)}{(1-ab)(1-b)} . \]
The denominator is 
\begin{equation*}
\begin{split}
& 1-(1+w^2)(1-w) -(1-w) +(1+w^2)(1-w)^2 \\
& = (1+w^2)[(1-w)^2-(1-w)] + w  \\
& = w^4 -w^3 + w^2 \\
& = w^2(1-w+w^2) . 
\end{split}
\end{equation*}

Hence, (\ref{ours}) becomes 
\begin{equation}
\frac{1}{(1-w^2)^2(1-w)^{r-1}}\left( \frac{(1-a^rb^r)(1-b) -(1-ab)(1-b^r)}{w^2(1-w+w^2)}  \right)  . 
\label{ourschanged}
\end{equation}
Comparing this to what we get from (\ref{shifted}) above, it suffices to show that
\begin{equation}
\frac{(1-a^rb^r)(1-b) -(1-ab)(1-b^r)}{w}
\label{compare1}
\end{equation}
is the same as 
\begin{equation}
\begin{split}
& - \left[ [(1+t^2)^r-1](1-t)^r +t(1-t)[(1-t)^r-1] \right] \\
& = -[(1+t^2)^r(1-t)^r-(1-t)^r +t(1-t)[(1-t)^r-1] ]\\
& = -(1+t^2)^r(1-t)^r +(1-t)^r -t(1-t)^{r+1} + t(1-t) . 
\end{split}
\label{compare2}
\end{equation}

We have that the numerator of (\ref{compare1}) is 
\begin{equation*}
\begin{split}
& 1-a^rb^r-b + a^r b^{r+1} -(1-ab-b^r+ab^{r+1}) \\
& = -a^r b^r -b + a^r b^{r+1}+ab + b^r-ab^{r+1} \\
& = a^r(b^{r+1}-b^r) + (a-1)b +b^r(1-ab) . 
\end{split}
\end{equation*}

The first term of this is 
\begin{equation} 
a^r(b^{r+1}-b^r) = a^rb^r (b-1) = -w(1+w^2)^r (1-w)^r . 
\label{term1}
\end{equation}
The second term is 
\begin{equation}
(a-1)b = w^2(1-w) . 
\label{term2}
\end{equation}
The third term is 
\begin{equation}
\begin{split}
b^r(1-ab) & = (1-w)^r[1-(1+w^2)(1-w)] \\
& = (1-w)^r[w-w^2  + w^3] \\
& = w(1-w)^r -w^2(1-w)^{r+1} . 
\end{split}
\label{term3}
\end{equation}

Taking the sum of (\ref{term1}), (\ref{term2}), (\ref{term3}) and dividing by $w$, we precisely get (\ref{compare2}) when we equate $w$ and $t$. 
\end{proof}

\end{document}